\documentclass[reqno, pdftex]{amsart}

\usepackage{amssymb,amstext,amsmath,amscd,amsthm,amsfonts,bm,mathtools,stmaryrd,enumerate,graphicx,latexsym}
\usepackage{tikz-cd}
\usepackage{hyperref}

\newtheorem{thm}{Theorem}[section]
\newtheorem{cor}[thm]{Corollary}
\newtheorem{lem}[thm]{Lemma}
\newtheorem{prop}[thm]{Proposition}

\theoremstyle{definition}
\newtheorem{dfn}[thm]{Definition}
\newtheorem{rem}[thm]{Remark}
\newtheorem{ques}[thm]{Question}

\newtheorem{ex}[thm]{Example}

\theoremstyle{remark}
\newtheorem*{ac}{Acknowledgments}

\def\Hom{\operatorname{Hom}}
\def\Ext{\operatorname{Ext}}
\def\Ann{\operatorname{Ann}}
\def\grade{\operatorname{grade}}
\def\depth{\operatorname{depth}}
\def\height{\operatorname{ht}}
\def\Gdim{\operatorname{Gdim}}

\def\pd{\operatorname{pd}}
\def\Ker{\operatorname{Ker}}
\def\Coker{\operatorname{Coker}}
\def\Im{\operatorname{Im}}
\def\K{\operatorname{K}}

\def\x{\bm{x}}

\def\M{\mathcal{M}}

\def\m{\mathfrak{m}}

\begin{document}
\allowdisplaybreaks
\title[Modules whose minimal free resolutions are self-dual or eventually periodic]{Modules whose minimal free resolutions are self-dual or eventually periodic}
\author{Shinnosuke Kosaka}
\address{Graduate School of Mathematics, Nagoya University, Furocho, Chikusaku, Nagoya 464-8602, Japan}
\email{kosaka.shinnosuke.d2@s.mail.nagoya-u.ac.jp}

\begin{abstract}
Let $R$ be a commutative noetherian local ring. In this paper, we study the self-duality and eventual periodicity of minimal free resolutions of finitely generated $R$-modules in terms of their syzygy modules and Ext modules. As an application, we recover theorems of Dey.
\end{abstract}

\maketitle

\section{Introduction}
Throughout this paper, we assume that all rings are commutative and noetherian, and that all modules are finitely generated. Studying the structure of minimal free resolutions of modules is one of the most important themes in commutative algebra. In this paper, we focus on the self-duality and eventual periodicity of those resolutions.

Let $R$ be a local ring with residue field $k$ and $M$ an $R$-module. Take a minimal free resolution $F=(\cdots\overset{d_3}\to F_2\overset{d_2}\to F_1\overset{d_1}\to F_0\to0)$
of $M$. For each integer $n\ge0$, the cokernel of $d_{n+1}$ is called the {\em $n$-th syzygy} of $M$ and denoted by $\Omega_R^nM$. We say that $F$ is {\em self-dual} if $\pd_RM<\infty$ and $F$ is isomorphic to the complex $F^*=(0\to F_0^*\overset{d_1^*}\to F_1^*\overset{d_2^*}\to\cdots\overset{d_{r-1}^*}\to F_{r-1}^*\overset{d_r^*}\to F_r^*\to0)$, where $r=\pd_RM$ and $(-)^*=\Hom_R(-,R)$. For an integer $a>0$, we say that $F$ is {\em eventually periodic} with period $a$ if there exists an integer $n\ge0$ such that $F$ is isomorphic to a complex of the form
\[\begin{array}{ccccccccccccc} (\cdots & \longrightarrow & F_{n+a-1} & \overset{d_{n+a-1}}\longrightarrow & F_{n+a-2} & \overset{d_{n+a-2}}\longrightarrow & \cdots & \overset{d_{n+2}}\longrightarrow & F_{n+1} & \overset{d_{n+1}}\longrightarrow & F_{n} & &\\ & \longrightarrow & F_{n+a-1} & \overset{d_{n+a-1}}\longrightarrow & F_{n+a-2} & \overset{d_{n+a-2}}\longrightarrow & \cdots & \overset{d_{n+2}}\longrightarrow & F_{n+1} & \overset{d_{n+1}}\longrightarrow & F_{n} & &\\ & \overset{d_n}\longrightarrow & F_{n-1} & \overset{d_{n-1}}\longrightarrow & F_{n-2} & \overset{d_{n-2}}\longrightarrow & \cdots & \overset{d_2}\longrightarrow & F_1 & \overset{d_1}\longrightarrow & F_0 & \longrightarrow & 0). \end{array}\]
This is equivalent to the existence of an integer $n$ such that $\Omega_R^nM\cong\Omega_R^{n+a}M$. We say that $M$ is {\em totally self-dual} if it is nonzero and satisfies:
\[
\Ext^i_R(M,R)\cong
\begin{cases}
M & \text{if $i=\grade M$,}\\
0 & \text{if $i\ne\grade M$.}
\end{cases}
\]

Regarding the self-duality of minimal free resolutions, we obtain the following theorem, which is the first main result of this paper.

\begin{thm}[Proositions \ref{prop mfr sd ts}, \ref{prop mrf sd syz} and Theorem \ref{thm mfr sd syz}]\label{main thm sd}
Let $R$ be a local ring, $M$ a nonzero $R$-module of grade $g$, and $F$ a minimal free resolution of $M$.
\begin{enumerate}[\rm(1)]
\item
Suppose that $M$ has finite projective dimension. Then $F$ is self-dual if and only if $M$ is totally self-dual. When this is the case, one has $(\Omega^i_RM)^*\cong\Omega^{g+1-i}_RM$ for all integers $0\le i\le g-1$.
\item 
Suppose that there exist two integers $0\le m,n<g$ such that $(\Omega^m_RM)^*\cong\Omega^n_RM$. Then one has that $m,n\ge2$, that $g=\pd_RM=m+n-1$, and that $F$ is self-dual.
\end{enumerate}
\end{thm}

Letting $M=k$ in the above theorem, we recover a theorem of Dey \cite[Theorem 1.1]{D}.

\begin{cor}[Dey]\label{dey sd}
Let $R$ be a local ring with $t=\depth R\ge3$. If $(\Omega_R^ik)^*\cong\Omega_R^ik$ for some integer $2\le i<t$, then $R$ is regular and $t=2i-1$.
\end{cor}

Regarding the eventual periodicity of minimal free resolutions, we obtain the following theorem, which is the second main result of this paper.

\begin{thm}[Theorems \ref{thm mfr ep} and \ref{thm mfr ep2}]\label{main thm ep}
Let $R$ be a local ring, $M$ a G-perfect $R$-module of grade $g$, and $F$ a minimal free resolution of $M$. Set $N=\Ext_R^g(M,R)$.
\begin{enumerate}[\rm(1)]
\item
Suppose that $F$ is eventually periodic with period $a>0$. Then one has $(\Omega_R^iN)^*\approx\Omega_R^{g+ma-i}M$ for all integers $i\ge g$ and $m\ge\frac{i}{a}$.
\item 
Suppose that there exist two integers $i\ge g$ and $j\ge0$ such that $i+j>g$ and $(\Omega_R^iN)^*\approx\Omega_R^jM$. Then $F$ is eventually periodic with period $i+j-g$.
\end{enumerate}
\end{thm}

Here, a module $M$ over a ring $R$ is called {\em G-perfect} if its Gorenstein dimension coincides with $\grade_RM$. For $R$-modules $M$ and $N$, we write $M\approx N$ if there exist projective $R$-modules $P,Q$ such that $M\oplus P\cong N\oplus Q$.

Letting $M=k$ in Theorem \ref{main thm ep}, we recover another theorem of Dey \cite[Theorem 1.2]{D}.

\begin{cor}[Dey]\label{dey ep}
Let $R$ be a local ring with $t=\depth R\ge2$.
\begin{enumerate}[\rm(1)]
\item 
If $R$ is a hypersurface and $t$ is even, then one has $(\Omega_R^ik)^*\cong\Omega_R^ik$ for all integers $i\ge t$.
\item
If $(\Omega_R^ik)^*\cong\Omega_R^ik$ for some integer $i\ge t$, then $R$ is a hypersurface.
\end{enumerate}
\end{cor}

In Section 2, we present some examples of totally self-dual modules and investigate their basic properties. In Section 3, we shall prove Theorems \ref{main thm sd} and \ref{main thm ep}.

\section{Totally self-dual modules}
Throughout this section, let $R$ be a ring. We begin with the definition of totally self-dual modules.

\begin{dfn}
Let $M$ be a nonzero $R$-module of grade $g$. We say that $M$ is {\em totally self-dual} if
\[
\Ext^i_R(M,R)\cong
\begin{cases}
M & \text{if $i=g$},\\
0 & \text{if $i\ne g$}.
\end{cases}
\]
\end{dfn}

We state some basic properties of totally self-dual modules.

\begin{lem}\label{ts property}
\begin{enumerate}[\rm(1)]
\item 
If $M$ and $N$ are totally self-dual $R$-modules of the same grade, then $M\oplus N$ is also totally self-dual.
\item 
Let $R\to S$ be a flat ring homomorphism and $M$ an $R$-module. If $M$ is totally self-dual and $M\otimes_RS$ is nonzero, then $M\otimes_RS$ is totally self-dual as an $S$-module. The converse holds if $R$ is local and $R\to S$ is the completion map.
\item
Let $M$ be an $R$-module and $x\in\Ann_RM$ an $R$-regular element. Then $M$ is totally self-dual as an $R$-module if and only if $M$ is totally self-dual as an $R/(x)$-module.
\item 
Let $M$ be a totally self-dual $R$-module of grade $g$ and $x\in R$ an $M$-regular element such that $M\ne xM$. Then $M/xM$ is a totally self-dual $R$-module of grade $g+1$.
\end{enumerate}
\end{lem}

\begin{proof}
(1) Clear.

(2) Put $g=\grade M$. If $M$ is totally self-dual, then we have
\[
\Ext_S^i(M\otimes_RS,S)\cong\Ext_R^i(M,R)\otimes_RS\cong
\begin{cases}
M\otimes_RS & \text{if $i=g$},\\
0 & \text{if $i\ne g$}.
\end{cases}
\]
Hence, $M\otimes_RS$ is totally self-dual as an $S$-module. Next, we assume that $R$ is local, that $R\to S$ is the completion map, and that $M\otimes_RS$ is totally self-dual as an $S$-module. Then we get
\[
\Ext_R^i(M,R)\otimes_RS\cong\Ext_S^i(M\otimes_RS,S)\cong
\begin{cases}
M\otimes_RS & \text{if $i=g$},\\
0 & \text{if $i\ne g$}.
\end{cases}
\]
By \cite[Corollary 1.15]{L}, $M$ is totally self-dual.

(3) By \cite[Lemma 18.2]{Mat}, we have
\[
\Ext^i_R(M,R)\cong
\begin{cases}
\Ext^{i-1}_{R/(x)}(M,R/(x)) & \text{if $i>0$},\\
0 & \text{if $i=0$}.
\end{cases}
\]
The assertion follows from this.

(4) From the short exact sequence $0\to M\overset{x}\to M\to M/xM\to0$, we obtain an exact sequence $\Ext_R^{i-1}(M,R)\to\Ext_R^i(M/xM,R)\to\Ext_R^i(M,R)$ for each $i$. If $i\ne g,g+1$, then we have $\Ext_R^{i-1}(M,R)=\Ext_R^i(M,R)=0$, and therefore $\Ext_R^i(M/xM,R)=0$. Moreover, there is an exact sequence
\[0\to\Ext_R^g(M/xM,R)\to\Ext_R^g(M,R)\overset{x}\to\Ext_R^g(M,R)\to\Ext_R^{g+1}(M/xM,R)\to0.\]
Since $\Ext_R^g(M,R)\cong M$ and $x$ is a regular element of $M$, we see that $\Ext_R^g(M/xM,R)=0$ and that $\Ext_R^{g+1}(M/xM,R)$ is isomorphic to $M/xM$. Thus, $M/xM$ is a totally self-dual $R$-module of grade $g+1$.
\end{proof}

The following are examples of totally self-dual modules.

\begin{ex}\label{ex ts}
\begin{enumerate}[\rm(1)]
\item 
It is obvious that every nonzero free $R$-module is totally self-dual.
\item
Let $R$ be a local ring with residue field $k$. Then $k$ is totally self-dual if and only if $R$ is Gorenstein by \cite[Theorem 18.1]{Mat}.
\item 
Let $R$ be a principal ideal domain. Then every nonzero torsion $R$-module $M$ is of the form $R/(a_1)\oplus R/(a_2)\oplus\cdots\oplus R/(a_n)$ such that $a_1,a_2,\dots,a_n$ are neither zero nor units. Hence, $M$ is totally self-dual by (1) and (4) of Lemma \ref{ts property}.
\item 
Let $R$ be a Gorenstein local ring and $I$ a proper ideal of $R$ such that $R/I$ is Cohen-Macaulay. Set $g=\grade I=\height I$. Then $\Ext_R^i(R/I,R)=0$ for every integer $i\ne g$ and $\Ext_R^g(R/I,R)$ is a canonical module of $R/I$ by \cite[Theorem 3.3.7]{BH}. Therefore, $R/I$ is totally self-dual as an $R$-module if and only if $R/I$ is Gorenstein.
\end{enumerate}
\end{ex}

Next, we examine properties of totally self-dual modules in terms of Gorenstein dimension. We first recall the definitions of Gorenstein dimension and related notions. 

\begin{dfn}
Let $M$ be an $R$-module.
\begin{enumerate}[\rm(1)]
\item
We say that $M$ is totally reflexive if it satisfies the following two conditions:
\begin{enumerate}[\rm(a)]
\item 
$M$ is reflexive. That is, the natural homomorphism $M\to M^{**}$ is isomorphic.
\item 
$\Ext_R^i(M,R)=\Ext_R^i(M^*,R)=0$ for all positive integers $i$.
\end{enumerate}
\item 
The {\em Gorenstein dimension} (or {\em G-dimension}) of $M$ is the infimum of the lengths of resolutions of $M$ by totally reflexive modules, and it is denoted by $\Gdim_RM$. That is,
\[
\Gdim_RM\coloneq\inf\left\{n\,\middle\vert
\begin{array}{c}
\text{There exists an exact sequence $0\to G_n\to\cdots\to G_0\to M\to0$}\\
\text{of $R$-modules such that $G_0,\dots,G_n$ are totally reflexive.}
\end{array}
\right\}.
\]
\item 
We say that $M$ is {\em G-perfect} if $M\ne0$ and $\grade_RM=\Gdim_RM$. We denote the category of G-perfect $R$-modules of grade $g$ by $\M_g(R)$.
\end{enumerate}
\end{dfn}

Note that the inequalities $\grade_RM\le\Gdim_RM\le\pd_RM$ hold for every nonzero $R$-module $M$. Therefore, if $M$ is perfect or totally reflexive, then $M$ is G-perfect.

The category $\M_g(R)$ has the following property.

\begin{prop}\label{gperf duality}
Let $g\ge0$ be an integer. Then the contravariant functor $\Ext_R^g(-,R)$ provides a duality on $\M_g(R)$. That is, the following two conditions hold for every $M\in\M_g(R)$:
\begin{enumerate}[\rm(1)]
\item 
$\Ext_R^g(M,R)\in\M_g(R)$.
\item
One has a natural isomorphism $\Ext_R^g(\Ext_R^g(M,R),R)\cong M$.
\end{enumerate}
\end{prop}

\begin{proof}
(1) This is shown in \cite[Theorem 1.8]{YKS}.

(2) Let $\x=x_1,\dots,x_g$ be an $R$-sequence in $\Ann_R(M)\subset\Ann_R(\Ext_R^g(M,R))$. Then $M$ is totally reflexive as an $R/(\x)$-module by \cite[Proposition 2.2.8]{C}. Therefore, by \cite[Lemma 18.2]{Mat}, we have natural isomorphisms $\Ext_R^g(\Ext_R^g(M,R),R)\cong\Hom_{R/(\x)}(\Hom_{R/(x)}(M,R/(\x)),R/(\x))\cong M$.
\end{proof}

The following proposition shows that totally self-dual modules are exactly the fixed points of $\Ext_R^g(-,R)$ in $\M_g(R)$.

\begin{prop}
Let $M$ be a nonzero $R$-module of grade $g$. Then the following two conditions are equivalent.
\begin{enumerate}[\rm(1)]
\item
$M$ is totally self-dual.
\item
$M$ is $G$-perfect and $\Ext_R^g(M,R)\cong M$.
\end{enumerate}
\end{prop}

\begin{proof}
(1)$\Rightarrow$(2): By \cite[Theorem 1.2.7]{C}, it suffices to show that $M$ has finite G-dimension. Let $\x=x_1,\dots x_g$ be an $R$-sequence in $\Ann_RM$. Then $M$ is totally self-dual as an $R/(\x)$-module by Lemma \ref{ts property}. By virtue of \cite[Proposition 2.2.8]{C}, it is enough to show that $M$ is totally reflexive as an $R/(\x)$-module. Therefore, we may assume that $g=0$. Since $M^*=\Ext^0_R(M,R)\cong M$, we have $M\cong M^{**}$. Hence, $M$ is reflexive by \cite[Proposition 1.1.9 (b)]{C}. Moreover, we see that $\Ext^i_R(M^*,R)\cong\Ext^i_R(M,R)=0$ for all integers $i>0$. Thus, $M$ is totally reflexive.

(2)$\Rightarrow$(1): The implication is clear by \cite[Theorem 1.2.7]{C}.
\end{proof}

We close the section by giving some more examples of totally self-dual modules.

\begin{ex}\label{ex ts2}
\begin{enumerate}[\rm(1)]
\item
Let $f\in R$ be a regular element which is not a unit. Set $\overline{(-)}=(-)\otimes_RR/(f)$. A pair of square matrices $(\varphi,\psi)$ of the same size is called a {\em matrix factorization} of $f$ if $\varphi\psi=\psi\varphi=fE$, where $E$ is the identity matrix. For a matrix factorization $(\varphi,\psi)$ of $f$, there is a periodic exact sequence
\[\cdots\overset{\overline{\psi}}\to\overline{R}^{\oplus n}\overset{\overline{\varphi}}\to\overline{R}^{\oplus n}\overset{\overline{\psi}}\to\overline{R}^{\oplus n}\overset{\overline{\varphi}}\to\overline{R}^{\oplus n}\overset{\overline{\psi}}\to\cdots\]
of free $\overline{R}$-modules; see \cite[Chapter 7]{Y} for details. Put $M=\Coker\varphi\cong\Coker\overline{\varphi}$. For a matrix $A$, we denote its transpose by $A^T$. Since $(\varphi^T,\psi^T)$ is also a matrix factorization of $f$, the complex
\[\cdots\overset{\overline{\psi^T}}\to\overline{R}^{\oplus n}\overset{\overline{\varphi^T}}\to\overline{R}^{\oplus n}\overset{\overline{\psi^T}}\to\overline{R}^{\oplus n}\overset{\overline{\varphi^T}}\to\overline{R}^{\oplus n}\overset{\overline{\psi^T}}\to\cdots\]
is also exact. This means that $M$ is a totally reflexive $\overline{R}$-module. Moreover, $M$ is a perfect $R$-module of grade $1$ with $\Ext_R^1(M,R)\cong\Hom_{\overline{R}}(M,\overline{R})\cong \Coker\overline{\varphi^T}\cong\Coker\varphi^T$. Hence, $M$ is totally self-dual as an $R$-module (or as an $\overline{R}$-module) if and only if $\Coker\varphi^T$ is isomorphic to $M$. This condition is satisfied if $\varphi^T$ is equivalent to $\varphi$. The converse holds if $R$ is local by the uniqueness of a minimal free resolution. Here, we say that two matrices $\varphi,\psi$ of the same size are {\em equivalent} if there exist invertible matrices $\alpha,\beta$ such that $\alpha\varphi=\psi\beta$.
\item 
Let $x$ be an element of $R$ such that $\underset{R}{(0:x)}=(x)$. Then there is an exact sequence
\[\cdots\overset{x}\to R\overset{x}\to R\overset{x}\to R\overset{x}\to\cdots.\]
By the same argument as (1), $R/(x)$ is a totally self-dual $R$-module.
\end{enumerate}
\end{ex}

\section{Proof of main results}

In this section, we prove our main results. Throughout this section, let $(R,\m,k)$ be a local ring. First, we show the relationship between totally self-dual modules and the self-duality of minimal free resolutions.

\begin{prop}\label{prop mfr sd ts}
Let $M$ be a nonzero $R$-module of finite projective dimension and $F$ a minimal free resolution of $M$. Then the following conditions are equivalent.
\begin{enumerate}[\rm(1)]
\item 
$F$ is self-dual.
\item 
$M$ is totally self-dual.
\end{enumerate}
\end{prop}

\begin{proof}
Put $n=\pd_RM$ and $F=(0\to F_n\overset{d_n}\to\cdots\overset{d_1}\to F_0\to0)$.

(1)$\Rightarrow$(2): Since the sequence
\[0\to F_0^*\overset{d_1^*}\to\cdots\overset{d_n^*}\to F_n^*\to M\to0\]
is exact, we have
\[
\Ext^i_R(M,R)\cong
\begin{cases}
M & \text{if $i=g$},\\
0 & \text{if $i\ne g$}.
\end{cases}
\]
Therefore, $M$ is totally self-dual.

(2)$\Rightarrow$(1): By the definition of a totally self-dual module, the sequence
\[0\to F_0^*\overset{d_1^*}\to\cdots\overset{d_n^*}\to F_n^*\to M\to0\]
is exact. Since $\Im d_i\subset\m F_{i-1}$, we get $\Im d_i^*\subset\m F_i^*$. Hence, $F^*=(0\to F_0^*\overset{d_1^*}\to\cdots\overset{d_n^*}\to F_n^*\to0)$ is a minimal free resolution of $M$. By the uniqueness of a minimal free resolution, $F^*$ is isomorphic to $F$.
\end{proof}

\begin{ex}\label{ex mfr sd}
\begin{enumerate}[\rm(1)]
\item
Let $\x=x_1,\dots,x_n$ be an $R$-regular sequence. Then the Koszul complex $\K(\x)$ of $\x$ is a minimal free resolution of $R/(\x)$. By \cite[Proposition 1.6.10]{BH}, $\K(\x)$ is self-dual. This also follows from Lemma \ref{ts property}(4) and Proposition \ref{prop mfr sd ts}.
\item 
The local ring $R$ is regular if and only if the minimal free resolution of $k$ is self-dual by Example \ref{ex ts}(2) and Proposition \ref{prop mfr sd ts}. This also follows from (1).
\end{enumerate}
\end{ex}

Next, we show that syzygies of modules with self-dual minimal free resolutions satisfy certain isomorphisms.

\begin{prop}\label{prop mrf sd syz}
Let $M$ be a nonzero $R$-module of grade $g$. If a minimal free resolution of $M$ is self-dual, then one has $(\Omega^i_RM)^*\cong\Omega^{g+1-i}_RM$ for all integers $0\le i\le g-1$.
\end{prop}

\begin{proof}
Let $0\to F_g\overset{d_g}\to\cdots\overset{d_2}\to F_1\overset{d_1}\to F_0\to0$ be a minimal free resolution of $M$. Then the $R$-dual complex $0\to F_0^*\overset{d_1^*}\to F_1^*\overset{d_2^*}\to\cdots\overset{d_g^*}\to F_g^*\to0$ is also a minimal free resolution of $M$. Therefore, for each integer $0\le i\le g-1$, there are isomorphisms $(\Omega^i_RM)^*\cong(\Coker d_{i+1})^*\cong\Ker(d_{i+1}^*)\cong\Omega^{g+1-i}_RM$.
\end{proof}

We consider whether the converse of the above proposition holds. More precisely, we raise the following question.

\begin{ques}\label{ques}
Let $M$ be a nonzero $R$-module of grade $g$ and $0\le i\le g-1$ an integer. Does the existence of an isomorphism $(\Omega^i_RM)^*\cong\Omega^{g+1-i}_RM$ imply the self-duality of a minimal free resolution of $M$?
\end{ques}

The following theorem shows that Question \ref{ques} is affirmative for all integers $2\le i\le g-1$.

\begin{thm}\label{thm mfr sd syz}
Let $M$ be a nonzero $R$-module of grade $g$. Suppose that there exist two integers $0\le m,n<g$ such that $(\Omega^m_RM)^*\cong\Omega^n_RM$. Then one has that $m,n\ge2$, that $g=\pd_RM=m+n-1$, and that a minimal free resolution of $M$ is self-dual.
\end{thm}

\begin{proof}
Let $F=(\cdots\overset{d_2}\to F_1\overset{d_1}\to F_0\to0)$ be a minimal free resolution of $M$. Since $\Ext_R^i(M,R)=0$ for every integer $0\le i\le m$, the sequence
\[0\to F_0^*\to\cdots\to F_{m-1}^*\to(\Omega_R^mM)^*\to0\]
is exact. Splicing this exact sequence with the exact sequence
\[0\to\Omega_R^nM\to F_{n-1}\to\cdots\to F_0\to M\to0,\]
we obtain an exact sequence
\[0\to F_0^*\to\cdots\to F_{m-1}^*\to F_{n-1}\to\cdots\to F_0\to M\to0.\]
This is a minimal free resolution of $M$. Hence, the projective dimension of $M$ is $m+n-1$. Since $m,n<g\le\pd_RM= m+n-1$, we have $m,n\ge2$. By \cite[Theorem 43]{Mas}, $\Omega_R^mM$ is reflexive. Therefore, $(\Omega_R^nM)^*\cong((\Omega_R^mM)^*)^*\cong\Omega_R^mM$. We may assume that $m\ge n$. By the uniqueness of a minimal free resolution, there is an isomorphism of exact sequences
\[(0\to F_0^*\to\cdots\to F_{m-1}^*\to(\Omega_R^mM)^*\to0)\cong(0\to F_{m+n-1}\to\cdots\to F_n\to\Omega_R^nM\to0).\]
Dualizing this isomorphism, we obtain an isomorphism of exact sequences
\[(0\to\Omega_R^mM\to F_{m-1}\to\cdots\to F_0)\cong(0\to(\Omega_R^nM)^*\to F_{n}^*\to\cdots\to F_{m+n-1}^*).\]
Since $m\ge n$, by combining the above two isomorphisms of exact sequences, we conclude that $F$ is isomorphic to the complex $(0\to F_0^*\overset{d_1^*}\to\cdots\overset{d_n^*}\to F_n^*\to0)$ and that $g=\pd_RM$.
\end{proof}

\begin{proof}[Proof of Corollary \ref{dey sd}]
The assertion immediately follows from Example \ref{ex mfr sd}(2) and Theorem \ref{thm mfr sd syz}.
\end{proof}

\begin{rem}
Question \ref{ques} is negative if $i=0$ or $i=1$. We have counterexamples in each case.

(The case $i=0$) Let $S=k\llbracket x,y\rrbracket$ be the formal power series ring over a field $k$. Consider the matrices
\[\varphi=
\begin{pmatrix}
x^2 & -y\\
0 & x
\end{pmatrix},
\quad\psi=
\begin{pmatrix}
x & y\\
0 & x^2
\end{pmatrix}
.\]
Then $(\varphi,\psi)$ is a matrix factorization of $x^3$, $\Coker\varphi\cong(x,y)/(x^3)$, and $\Coker\psi\cong(x^2,y)/(x^3)$. Put $R=S/(x^3)$, $M=(x,y)/(x^3)$, and $N=(x^2,y)/(x^3)$. By elementary operations, we see that the transpose of $\varphi$ is equivalent to $\psi$. Hence, $\Hom_R(M,R)\cong N$. From Example \ref{ex ts2}(1), $M$ is a perfect $S$-module of grade $1$ with $\Ext_S^1(M,S)\cong N$. Since $(\Omega_S^0M)^*=\Ext_S^0(M,S)=0=\Omega_S^2M=\Omega_S^{1+1-0}M$, $M$ satisfies the assumption of Question \ref{ques}. Assume that $M$ is a totally self-dual $S$-module. Then $M\cong N$. Taking the tensor products with $S/(x^2,y^2)$, we obtain exact sequences
\[\begin{tikzcd}[ampersand replacement=\&]
(S/(x^2,y^2))^{\oplus2}\ar[r, "{\begin{pmatrix}0&-y\\0&x\end{pmatrix}}"] \& (S/(x^2,y^2))^{\oplus2}\ar[r] \& M/(x^2,y^2)M\ar[r] \& 0
\end{tikzcd},\]
\[\begin{tikzcd}[ampersand replacement=\&]
(S/(x^2,y^2))^{\oplus2}\ar[r, "{\begin{pmatrix}x&y\\0&0\end{pmatrix}}"] \& (S/(x^2,y^2))^{\oplus2}\ar[r] \& N/(x^2,y^2)N\ar[r] \& 0
\end{tikzcd}.\]
Hence, $N/(x^2,y^2)N\cong k\oplus S/(x^2,y^2)$. On the other hand, $M/(x^2,y^2)M$ has no $S/(x^2,y^2)$-free summand, since $xy$ kills $(M/(x^2,y^2)M)$. This contradicts the assumption that $M$ is isomorphic to $N$. Thus, $M$ is not totally self-dual as an $S$-module.

(The case $i=1$) We use the same notation as above. Set $R'=R/(y^2)=S/(x^3,y^2)$, $M'=M/y^2M$, and $N'=N/y^2N$. Since $y^2$ is an $R$-regular element, the sequences
\[\cdots\overset{\psi\otimes_SR'}\longrightarrow R'^{\oplus 2}\overset{\varphi\otimes_SR'}\longrightarrow R'^{\oplus 2}\overset{\psi\otimes_SR'}\longrightarrow R'^{\oplus 2}\overset{\varphi\otimes_SR'}\longrightarrow R'^{\oplus 2}\overset{\psi\otimes_SR'}\longrightarrow\cdots,\]
\[\cdots\overset{\psi^T\otimes_SR'}\longrightarrow R'^{\oplus 2}\overset{\varphi^T\otimes_SR'}\longrightarrow R'^{\oplus 2}\overset{\psi^T\otimes_SR'}\longrightarrow R'^{\oplus 2}\overset{\varphi^T\otimes_SR'}\longrightarrow R'^{\oplus 2}\overset{\psi^T\otimes_SR'}\longrightarrow\cdots,\]
are exact, $\Coker(\varphi\otimes_SR')\cong M'$, and $\Coker(\psi\otimes_SR')\cong N'$. Hence $M'$ is a perfect $S$-module of grade $2$ with $\Ext_S^2(M',S)\cong\Hom_{R'}(M',R')\cong N'$. Let $0\to S^{\oplus c}\to S^{\oplus b}\to S^{\oplus a}\to0$ be a minimal $S$-free resolution of $M'$. Then $a=\mu_S(M')=\mu_{R'}(M')=2$ and $c=\mu_S(\Ext_S^2(M',S))=\mu_{R'}(N')=2$. Here, for a module $L$ over a local ring $T$, $\mu_T(L)$ denotes the minimal number of generators of $L$. As $(\Omega_S^1M')^*\cong S^{\oplus a}=S^{\oplus2}=S^{\oplus c}\cong\Omega_S^2M'=\Omega_S^{2+1-1}M'$, $M'$ satisfies the assumption of Question \ref{ques}. Since $M/(x^2,y^2)M\ncong N/(x^2,y^2)N$, $M'$ is not isomorphic to $N'$. Therefore, $M'$ is not totally self-dual as an $S$-module.
\end{rem}

Next, we consider the eventual periodicity of minimal free resolutions of G-perfect modules. The following lemma is the key to prove our main theorem.

\begin{lem}\label{lem dual st iso}
Let $M$ be a nonzero $R$-module of grade $g$ such that $\Ext_R^i(M,R)=0$ for all nonnegative integers $i\ne g$, and set $N=\Ext_R^g(M,R)$. Then one has $\Omega_R^gN\approx\Omega_R^i((\Omega_R^iM)^*)$ for all integers $i\ge g$.
\end{lem}

\begin{proof}
Let $F=(\cdots\overset{d_2}\to F_1\overset{d_1}\to F_0\to0)$ be a minimal free resolution of $M$. By assumption, there are two exact sequences
\begin{equation}\label{ex seq lem 1}
0\to F_0^*\overset{d_1^*}\to\cdots\overset{d_{g-1}^*}\to F_{g-1}^*\to(\Omega_R^gM)^*\to N\to0,
\end{equation}
\begin{equation}\label{ex seq lem 2}
0\to(\Omega_R^gM)^*\to F_g^*\overset{d_{g+1}^*}\to\cdots\overset{d_{i-1}^*}\to F_{i-1}^*\to(\Omega_R^iM)^*\to0.
\end{equation}
Splitting \eqref{ex seq lem 1} into the two exact sequences $0\to F_0^*\overset{d_1^*}\to\cdots\overset{d_{g-1}^*}\to F_{g-1}^*\to L\to0$ and $0\to L\to(\Omega_R^gM)^*\to N\to0$, we have $\Omega_R^gL=0$. By the horseshoe lemma, we obtain a short exact sequence $0\to\Omega_R^gL\to\Omega_R^g((\Omega_R^gM)^*)\oplus R^{\oplus a}\to\Omega_R^gN\to0$ with $a\ge0$. We get $\Omega_R^g((\Omega_R^gM)^*)\approx\Omega_R^gN$.

By \eqref{ex seq lem 2}, we have $\Omega_R^{i-g}((\Omega_R^iM)^*)\cong(\Omega_R^gM)^*$. It follows that $\Omega_R^i((\Omega_R^iM)^*)\cong\Omega_R^g(\Omega_R^{i-g}((\Omega_R^iM)^*))\cong\Omega_R^g((\Omega_R^gM)^*)\approx\Omega_R^gN$.
\end{proof}

The following theorem shows that syzygies of G-perfect modules with eventually periodic minimal free resolutions satisfy a certain duality.

\begin{thm}\label{thm mfr ep}
Let $M$ be a G-perfect $R$-module of grade $g$, $F$ a minimal free resolution of $M$, and set $N=\Ext_R^g(M,R)$. Suppose that $F$ is eventually periodic with period $a>0$. Then one has $(\Omega_R^iN)^*\approx\Omega_R^{g+ma-i}M$ for all integers $i\ge g$ and $m\ge\frac{i}{a}$.
\end{thm}

\begin{proof}
By assumption, there exists an integer $n\ge0$ such that $\Omega_R^nM\cong\Omega_R^{n+a}M$. It follows from Lemma \ref{gperf duality} and Lemma \ref{lem dual st iso} that $\Omega_R^i((\Omega_R^iN)^*)\approx\Omega_R^gM$. Therefore, we have $\Omega_R^{n+i}((\Omega_R^iN)^*)\approx\Omega_R^{n+g}M\cong\Omega_R^{n+g+ma}M\cong\Omega_R^{n+i}(\Omega_R^{g+ma-i}M)$. Since $\Gdim_RM=\Gdim_RN=g$ and $ma\ge i\ge g$, the $R$-modules $(\Omega_R^iN)^*$ and $\Omega_R^{g+ma-i}M$ are totally reflexive. By \cite[Lemma 3.6]{D}, we get $(\Omega_R^iN)^*\approx\Omega_R^{g+ma-i}M$.
\end{proof}

We consider the converse of the above theorem. More precisely, we prove the following.

\begin{thm}\label{thm mfr ep2}
Let $M$ be a G-perfect $R$-module of grade $g$, $F$ a minimal free resolution of $M$, and set $N=\Ext_R^g(M,R)$. Suppose that there exist two integers $i\ge g$ and $j\ge0$ such that $i+j>g$ and $(\Omega_R^iN)^*\approx\Omega_R^jM$. Then $F$ is eventually periodic with period $i+j-g$.
\end{thm}

\begin{proof}
By Lemma \ref{gperf duality} and Lemma \ref{lem dual st iso}, we have $\Omega_R^gM\approx\Omega_R^i((\Omega_R^iN)^*)\approx\Omega_R^{i+j}M$. Therefore, $\Omega_R^{g+1}M\cong\Omega_R^{i+j+1}M=\Omega_R^{g+1+(i+j-g)}$. Hence, $M$ is eventually periodic with period $i+j-g$.
\end{proof}

Finally, we prove Corollary \ref{dey ep}. For this purpose, we prove the following lemma.

\begin{lem}\label{hs k ep}
Let $R$ be a local ring with residue field $k$. Then the minimal free resolution of $k$ is eventually periodic if and only if $R$ is a hypersurface.
\end{lem}

\begin{proof}
Suppose that the minimal free resolution of $k$ is eventually periodic. Then the Betti numbers of $k$ are bounded. By \cite[Theorem 8.1.2]{A}, $R$ is a hypersurface.

Suppose that $R$ is a hypersurface of dimension $d$. We may assume that $R$ is complete. Then $R$ is isomorphic to a ring $S/(f)$, where $S$ is a regular local ring and $f\in S\setminus\{0\}$ is not a unit. Since $\Omega_R^dk$ is maximal Cohen--Macaulay, there exists a matrix factorization $(\varphi,\psi)$ of $f$ such that $\Coker\varphi\cong\Omega_R^dk$; see \cite[Chapter 7]{Y} for details. Hence, the minimal free resolution of $\Omega_R^dk$ is periodic. It follows that the minimal free resolution of $k$ is eventually periodic.
\end{proof}

\begin{proof}[Proof of Corollary \ref{dey ep}]
By \cite[Lemma 3.2]{D}, $R$ is Gorenstein. Hence, $k$ is totally self-dual. The assertion follows from Theorems \ref{thm mfr ep}, \ref{thm mfr ep2}, and Lemma \ref{hs k ep}.
\end{proof}

\begin{ac}
The author would like to thank his Ph.D. advisor Ryo Takahashi for giving many thoughtful questions and helpful discussions. He also thanks Kaito Kimura and Yuki Mifune, and Yuya Otake for their valuable comments.
\end{ac}


\begin{thebibliography}{99}
\bibitem{A}
{\sc L. L. Avramov}, Infinite free resolutions, {\em Six lectures on commutative algebra (Bellaterra, 1996)}, 1–118, Progr. Math., 166, {\em Birkh\"{a}user, Basel}, 1998.
\bibitem{BH}
{\sc W. Bruns; J. Herzog}, Cohen--Macaulay rings, revised edition, Cambridge Studies in Advanced Mathematics {\bf 39}, {\it Cambridge University Press, Cambridge}, 1998.
\bibitem{C}
{\sc L. W. Christensen}, Gorenstein dimensions, Lecture Notes in Mathematics, 1747, Springer--Verlag, Berlin, 2000.
\bibitem{D}
{\sc S. Dey}, When are syzygies of the residue field self-dual?, preprint (2025), {\tt arXiv:2505.15707}.
\bibitem{L}
{\sc G. J. Leuschke; R. Wiegand}, Cohen–Macaulay representations, Math. Surveys Monogr. {\bf 181}, {\it American Mathematical Society, Providence, RI}, 2012.
\bibitem{Mas}
{\sc V. Ma\c{s}ek}, Gorenstein dimension and torsion of modules over commutative Noetherian rings, {\em Comm. Algebra} {\bf 20} (2000), no. 12, 5783--5812.
\bibitem{Mat}
{\sc H. Matsumura}, Commutative ring theory, Translated from the Japanese by M. Reid, Second edition, Cambridge Studies in Advanced Mathematics {\bf 8}, {\it Cambridge University Press, Cambridge}, 1989.
\bibitem{YKS}
{\sc S. Yassemi; L. Khatami; T. Sharif}, Grade and Gorenstein Dimension, {\em Comm. Algebra} {\bf 29} (2001), no. 11, 5085--5094.
\bibitem{Y}
{\sc Y. Yoshino}, Cohen--Macaulay modules over Cohen--Macaulay rings, London Mathematical Society Lecture Note Series, {\bf 146}, {\it London Mathematical Society}, 1990.
\end{thebibliography}
\end{document}